\documentclass{amsart}
\usepackage[left=1in,top=1in,right=1in,bottom=1in]{geometry}
\usepackage[utf8x]{inputenc}
\usepackage{bbm}
\usepackage{amsmath,amsthm,amsfonts,amssymb,amscd}
\usepackage{psfrag,graphicx}

\usepackage{epic,eepic}
\usepackage{color}
\usepackage{ebezier}
\usepackage{enumerate}

\newtheorem*{conj*}{Conjecture}
\newtheorem*{cor*}{Corollary}
\newtheorem{theorem}{Theorem}[section]

\newtheorem{proposition}[theorem]{Proposition}
\newtheorem{corollary}[theorem]{Corollary}
\newtheorem{lemma}[theorem]{Lemma}

\theoremstyle{definition}
\newtheorem*{def*}{Definition}

\newtheorem{example}[theorem]{Example}

\newtheorem{definition}[theorem]{Definition}

\newcommand{\SC}{{\mathcal C}}

\newcommand{\SM}{{\mathcal M}}

\newcommand{\al} {\alpha}       
        
    \newcommand{\Ga}{\Gamma}
\newcommand{\de} {\delta}

\newcommand{\Z}{\mathbb{Z}}
\newcommand{\N}{\mathbb{N}}
\newcommand{\R}{\mathbb{R}}

\makeatletter
\newcommand{\tpitchfork}{
  \vbox{
    \baselineskip\z@skip
    \lineskip-.52ex
    \lineskiplimit\maxdimen
    \m@th
    \ialign{##\crcr\hidewidth\smash{$-$}\hidewidth\crcr$\pitchfork$\crcr}
  }
}
\makeatother

\thanks{2010 {\it Mathematics Subject Classification}.  37A35, 37C10, 37C40}

 \keywords{Rescaled expansive measures; Flows with singularities}

\theoremstyle{plain}

\newtheorem{Thm}{Theorem}[section]
\newtheorem{Lem}[Thm]{Lemma}
\newtheorem{Prop}[Thm]{Proposition}
\newtheorem{Cor}[Thm]{Corollary}
\newtheorem{Def}[Thm] {Definition}

\long\def\begcom#1\endcom{}

\def\ln{\operatorname{ln}}

\def\vep{\varepsilon}

\begin{document}

\title{Rescaled expansive measures for flows}

\author[Yun Yang]{Yun Yang$^{\dagger}$}
	\address{Yun Yang, Department of Mathematics, Virginia Polytechnic Institute and State University, Blacksburg, VA, United States 24060}
	\thanks{$^{\dagger}$Y.Y. is supported by a grant from the National Science Foundation (DMS-2000167).}

	\email{yunyang@vt.edu}

\date{}
\maketitle
\begin{abstract}
We introduce the notion of rescaled expansive measures to study a measure-theoretic formulation of rescaled expansiveness for flows, particularly in the presence of singularities. Equivalent definitions are established via reparametrizations of different regularities. Under the assumption of positive entropy, we prove the existence of invariant rescaled expansive measures. In the appendix, we derive a rescaled version of the Brin–Katok local entropy formula for flows, extending \cite{JCWZ} from nonsingular flows to general flows that may include singularities. This framework provides new tools for understanding entropy and expansiveness in continuous-time dynamical systems with singularities.

  \end{abstract}

\section{Introduction}
The presence of singularities has a significant impact on the dynamical properties of flows. For instance, flows with singularities may exhibit entropy degeneracy, a phenomenon that does not occur in flows without singularities. Analyzing the long-term behavior of flows with singularities is therefore substantially more challenging than for their nonsingular counterparts. In this paper, we investigate measure-theoretic expansiveness by scaling neighborhood sizes according to the local speed of the flow, leading to the notion of {\it rescaled expansive measures}.

Expansiveness is a classical property in the theory of chaotic dynamical systems, particularly for homeomorphisms, where it captures the idea that any two distinct points eventually separate by at least a fixed distance under iteration. This notion plays a central role in understanding sensitivity to initial conditions and the complexity of orbits.
For flows, however, defining expansiveness is more delicate due to the inherent ambiguity introduced by the flow direction: any two nearby points on the same trajectory remain arbitrarily close when shifted in time. To address this, several notions of expansiveness for flows have been developed, incorporating time reparametrizations to distinguish genuinely distinct orbits (\cite{BW, KS, Kom1, Ruggiero1996, WW}, among others).
In this work, we focus on the notion of rescaled expansiveness, introduced in \cite{WW}, and propose a new definition to study its measure-theoretic counterpart.

 A flow $\phi_t$ generated by a vector field $X$ is {\it rescaled expansive} on a compact invariant set $\Lambda$ if for any $\vep>0$ there is $\delta>0$ such that, for any $x,y\in \Lambda$ and any increasing continuous functions $h: \mathbb R\to \mathbb R$, if $$d(\phi_t(x), \phi_{h(t)}(y)\le \delta\|X(\phi_t(x))\|$$ for all $t\in \mathbb R$, then $\phi_{h(t)}(y)\in \phi_{[-\vep, \vep]}(\phi_t(x))$ for all $t\in \mathbb R$. 
 When the vector field has no singularities, rescaled expansiveness coincides with classical definitions such as those in \cite{BW, KS} (see Proposition 6.1 in \cite{WW}). The idea of rescaling by flow speed has its roots in the work of Liao on standard systems of differential equations \cite{L1, L2}, and has since been further developed and applied in numerous studies, including \cite{GY,HW,Y,SGW,WW}.

In this paper, we investigate a measure-theoretic formulation of rescaled expansiveness. This approach is motivated by the increasing interest in measure-theoretic analogs of classical dynamical properties, which often provide finer information about the typical behavior of systems.   



\begin{definition}
 For a flow $\phi$  generated by a $C^1$ vector field  $X$ on a compact manifold  $M$,  {\it the rescaled dynamical ball} centered in $x \in M$ with radius $\vep>0$ is defined as
 $$\Ga_\vep(x) = \{y \in M : \exists h \in \SC^0_0,\ \text{s.t.}\ d(\phi_s(x), \phi_{h(s)}(y))\leq \vep \|X(\phi_s(x))\| \enspace \forall s \in \R\}$$
 where $C^0_0$ is the set of continuous functions $h:\R\rightarrow\R$ with $h(0)=0.$
A Borel probability measure $\mu$ on $M$ is called {\it rescaled expansive} for the flow $\phi$ if there exists $\vep > 0$ such that $\mu(\Gamma_\vep(x)) = 0$ for every $x \in M$.
\end{definition}
We note here that when $\vep$ is sufficiently small, $\Ga_\vep(x)$ is the same as 
$$\Ga_\vep(x)=\{y \in M : \exists h \in \SC^0_0 \ \text{s.t.}\ d(\phi_{h(s)}(x), \phi_{s}(y))\leq \vep \|X(\phi_{h(s)}(x))\| \enspace \forall s \in \R\}.$$
This follows directly from Lemma \ref{increasing} by boosting $h$ to be an increasing homeomorphism and then applying $h^{-1}$. In Section 3, we explore several questions concerning the enhancement of the regularities for reparametrizations starting from  a merely continuous $h$.  We show not only how to boost to be increasing $Rep$, but also an ``almost identical" $Rep(\al)$ property can be obtained.
\begin{definition}Define
$$Rep=\{h\in C^0_0: h \text{ is an increasing homeomorphism}\} \text{ and }$$
$$Rep(\alpha)=\left\{h\in C^0_0: \left|\frac{h(s)-h(t)}{s-t} - 1\right|\leq \al \ \text{ for any } s\neq t \right\}.$$
Define {\it the rescaled dynamical balls} by reparametrizations with different regularities $Rep$ and $Rep(\al)$ to be:
$$ \Gamma_{\vep}^\al(x)=\{y\in X:\exists h\in Rep(\al)\enspace \text{s.t.} \enspace d(\phi_s(x),\phi_{h(s)}(y))\leq\varepsilon \|X(\phi_s(x))\|\enspace \forall s\in \mathbb{R}\} \text{ and }$$ $$ \Gamma_{\vep}^{\al,\ast}(x)=\{y\in X:\exists h\in Rep(\al)\cap Rep\enspace \text{s.t.} \enspace d(\phi_s(x),\phi_{h(s)}(y))\leq\varepsilon \|X(\phi_s(x))\|\enspace \forall s\in \mathbb{R}\}.$$
\end{definition}
 
 The following results concerning equivalent definitions of rescaled expansive measures tells us that the reparametrization in the definition of rescaled expansive measures can be strengthened to be  increasing and furthermore, almost identical.
\begin{theorem}[Equivalent definitions of rescaled expansive measures]\label{equivalent}
Let  $\phi$ be a flow generated by a $C^1$ vector field  $X$ on a compact manifold  $M$ and $\mu$ be a Borel probability measure on $M$. Then the following statements are equivalent:
    \begin{enumerate}
        \item\label{thmA_item1} $\mu$ is rescaled expansive for $\phi$;
          \item\label{thmA_item2} for every $\al \in (0,1)$ there exists $\vep>0$ such that $\mu(\Ga^\al_\vep(x))=0$ for for $\mu$-almost every $x \in M$;
        \item\label{thmA_item2} for every $\al \in (0,1)$ there exists $\vep>0$ such that $\mu(\Ga^{\al,\ast}_\vep(x))=0$ for $\mu$-almost every $x \in M$.
    \end{enumerate} 
\end{theorem}

Speaking of the measure-theoretic perspective of expansiveness,  another related concept was that of an \emph{expansive measure}, first introduced for discrete-time systems in \cite{MS}. Later on, \cite{CarrascoMorales2014} generalized it to continuous flows. 
We note here that (rescaled) expansive measures are not required to be invariant.  The existence of (rescaled) expansive invariant measures is far from guaranteed in general, and constructing such measures is significantly challenging.
 Indeed, in \cite{LMS} it was discussed an example of homeomorphism which is expansive, but does not admit any expansive measure. In \cite{AM}, it was shown that positive entropy is a sufficient condition for an ergodic invariant measure to be expansive, implying the existence of expansive invariant measures for every homeomorphism with positive topological entropy. 
Recently,  \cite{PRT} explored similar properties for flows without singularities. The following theorem generalizes the case from the singularity free flows to the general flows that might have singularities.  

Denote by $X$ a vector field on $M$. 
We call $x \in M$ a {\it singularity} of $X$ if $X(x) = 0$. Let ${\rm Sing}(X)$ denote the set of singularities  of $X$. See Definition 3.6 for the definition of positively rescaled expansive measures. 
\begin{theorem}[The existence of invariant rescaled expansive measures]\label{entropythm}
   Let  $\phi$ be a flow generated by a $C^1$ vector field $X$ on a compact manifold  $M$ with a Borel probability measure $\mu$ such that $\mu( {\rm Sing}(X))=0$.   If $\mu$ is an ergodic $\phi$-invariant measure with positive entropy, 
 then $\mu$ is (positively) rescaled expansive. In particular, if $\phi$ has positive topological entropy, then $\phi$ admits (positively) rescaled expansive invariant measures. 
\end{theorem}
 The majority of the related existing literature focuses on flows without singularities. 
The presence of singularities introduces substantial analytical subtleties. In this paper, we establish several foundational lemmas rigorously in the context of flows that include singularities.  To prove Theorem \ref{equivalent} and Theorem \ref{entropythm}, we rely on two main ingredients: control over reparametrizations and the lower bound of the decay rate of the rescaled Bowen balls. To develop these tools, the paper is organized as follows:
In Section 2, we establish several fundamental lemmas concerning separation properties based on the existence of flowboxes. These results lay the groundwork for Section 3, where we investigate different definitions of rescaled expansive measures with respect to reparametrization within $C^0_0, \text{Rep},$ and $\text{Rep}(\al)$ and then we finish the proof of Theorem \ref{equivalent}.
 In Section 4, we obtain a control on the lower bound of the decay rate of the rescaled Bowen balls and then apply this control to obtain Theorem \ref{entropythm}. Finally, in Appendix, we derive a rescaled version of the Brin–Katok local entropy formula for general flows that might have singularities.

\section{Flowbox and Separation}
In this section, we state a few important lemmas regarding the existence of flowboxes and some separation properties for general flows that might have singularities. 
The concept and term “flowbox” are well-established in the literature (\cite{PR}). Here we refer to a uniform relative version of the classical flowbox theorem, which was built in \cite{WW}.  

Let $M$ be a compact manifold and $X$ be a $C^1$ vector field on $M$. 
We call $x \in  M$ a {\it regular point} if $x \in M\setminus {\rm Sing}(X)$.  Denote by
$$T_xM(r)=\{v\in T_xM: \|v\|\le r\},  N_x=\{v\in T_xM: <v, X(x)>=0\}.$$
 By the compactness of $M$ and the $C^1$ smoothness of $X$,
there are constants $L>0$ and $r>0$ such that for any $x\in M$ the
vector fields
$\bar{X}=(\exp_x^{-1})_*(X|_{B_{r}(x)})$ in
$T_xM(r)$ are locally Lipschitz vector fields with Lipschitz constant $L$, which will be used many times in the proofs.  
For every regular point $x\in M$, denote by $$U_x(r\|X(x)\|)=\{v+tX(x)\in T_xM: v\in N_x, \|v\|\leq r\|X(x)\|, |t|\leq r\}$$ the tangent box of relative size $r$ at $x$. Define a $C^1$ map  $F_x:U_x(r\|X(x)\|)\to M$ to be $$F_x(v+tX(x))=\phi_t(\exp_x(v)).$$
\begin{Prop}[Proposition 2.3, \cite{WW}]\label{WW} For any $C^1$ vector field $X$ on $M$, there is $T_0>0$ such that for any regular point $x$ of ${X}$, $F_x:{U}_x(T_0\|X(x)\|)\to M$ is an embedding whose image contains no singularities  of $X$, and $m(D_pF_x)\geq 1/3$ and
$\|D_pF_x\|\leq3$ for every $p\in {U}_x(T_0\|X(x)\|)$.
\end{Prop}

\begin{lemma}[Lemma 2.1 in ~\cite{WY}] \label{2}
For any $C^1$ vector field $X$ on $M$,  there is a constant $c>0$ such that  for any two regular points $x,y\in M$ satisfying   $d(x,y)<c\|X(x)\|$, one has $$\frac{1}{2}\|X(x)\|\leq\|X(y)\|\leq 2\|X(x)\|.$$
\end{lemma}

\begin{Lem}[Lemma 2.3 in ~\cite{WW}]\label{3}
For any $C^1$ vector field $X$ on $M$, there exists a constant $T_0>0$   such that for any    regular points $x\in M$, one has the following properties:
\begin{itemize}
  \item [(1)]For any $0<\vep\leq \frac{T_0}{3}$ and $t\in[-T_0,T_0]$, if $d(x,\phi_t(x))\leq \vep\|X(x)\|$, then $|t|\leq3\vep$.
  \item [(2)]For any $0<\vep\leq \frac{T_0}{3}$, if  $~\phi_{[0,t]}(x)\subset B(x,\vep\|X(x)\|)$, then $|t|\leq3\vep$.
\end{itemize}
\end{Lem}

Based on the flowbox lemmas above, we obtain the following powerful separation properties.
\begin{Lem}[Separation property]\label{separation0}Let $\phi$ be a flow generated by a $C^1$ vector field $X$ on a compact space $M$. Then, there exists $T_0>0$ such that for any $0<T<T_0,$ there exists $\gamma>0$, such that 
$$d(\phi_t(x), x)\geq \gamma \|X(x)\|,$$
for any regular point $x\in M$ and $|t|\in [T,T_0]$.
\end{Lem}
\begin{proof}Take $T_0$ to be the number given in Lemma \ref{3}. We claim that  $T_0$ is what we want. Otherwise, there exists $T$ satisfying $0<T<T_0$ such that for any $n$ there exists $x_n\in M$ with 
$$d(\phi_{t_n}x_n,x_n)<\frac{1}{n}\|X(x_n)\|,$$
for some $|t_n|\in [T,T_0]$.  
By Lemma \ref{3}, we have control on the time $|t_n|\leq \frac{3}{n}$ for any sufficiently large $n$, which is a contradiction with the requirement that $|t_n|\in [T,T_0]$.
\end{proof}
\begin{Cor}\label{separation1}Let $\phi$ be a flow generated by a $C^1$ vector field $X$ on a compact space $M$. Then, there exists $T_0>0$ such that for any $0<T<T_0,$ there exists $\gamma>0$, such that 
$$d(\phi_t(x), y)\geq \gamma \|X(x)\|,$$
for any two regular points $x,y \in M$  with $d(x,y)\leq  \gamma \|X(x)\|$ and any $|t|\in[T,T_0]$.
\end{Cor}
\begin{proof}Let $T_0>0$ be the one given in Lemma \ref{separation0}. Then, for any $0<T<T_0,$ there exists $\gamma>0$, such that $d(\phi_t(x), x)\geq 2\gamma \|X(x)\|,$
for any regular point $x\in M$ and $|t|\in [T,T_0]$.
 Here we use $2\gamma$ to stand for the $\gamma$ in Lemma \ref{separation0}. 
Then for any two regular points $x,y \in M$  with $d(x,y)<\gamma$, we have
$$d(\phi_t(x),y)>d(\phi_t(x), x)-d(x,y)>2\gamma-\gamma=\gamma,$$
for any $|t|\in [T,T_0]$.
Hence the proof is completed. 
\end{proof}

\section{Regularity of reparametrizations}
Our goal in this section is to show some interesting properties regarding boosting the regularity of the reparameterizations for flows. Here by saying different regularity, we mean the reparametrization within $C^0_0, \text{Rep},$ and $\text{Rep}(\al)$ respectively. 
In particular, we shall build connections among reparametrizations with different regularities and moreover obtain some equivalent definitions of the rescaled expansive measures. Similar results were also obtained for the positively rescaled expansive measures.  

\subsection{From continuity to increasing continuity}
In this subsection, we obtain an increasing continuous reparametrization from a continuous reparametrization. A similar result was obtained in \cite{BW} (Lemma 2) and \cite{CarrascoMorales2014} (Lemma 2.3) for flows without singularities. Our results generalize this property to general flows that might have singularities. Due to the existence of singularities, the proof is significantly more difficult.
\begin{definition}
Let $\phi$ be a flow generated by a $C^1$ vector field $X$ on a compact space $M$. For $x \in M$ and $\vep>0$, denote by
$$ \Ga^{\ast}_{\vep}(x)=\{y\in X:\exists h\in Rep \enspace \text{s.t.} \enspace d(\phi_s(x),\phi_{h(s)}(y))\leq\varepsilon \|X(\phi_s(x))\|\enspace \forall s\in \mathbb{R}\}.$$
\end{definition}

\begin{lemma}\label{increasing}
  Let $\phi$ be a flow generated by a $C^1$ vector field $X$ on a compact space $M$.  Then for every $\vep>0$ there exists $\de>0$ such that $\Ga_\de(x) \subseteq \Ga^{\ast}_\vep(x)$ for any regular point $x \in M$.
\end{lemma}
\begin{proof}Let $L>0$ be the local Lipschitz constant of $X$.  By the continuous dependence of the solution on the initial value,  for any regular points $x\in M$ and $t\in \mathbb{R}$, we have $$e^{-L|t|}\|X(x)\|\leq\|X(\phi_{t}(x))\|\leq  e^{L|t|}\|X(x)\|.$$ Moreover, for $x,y\in M$, we have $e^{-L|t|}d(x,y)\leq d(\phi_{t}(x),\phi_{t}(y))\leq e^{L|t|}d(x,y)$. 

Let $T_0$ be the constant given in Corollary \ref{separation1}. We claim that
 for any $0<T<T_0$, there are $\delta_T>0$ such that for any two regular points $x$ and $y$ if there exists $h\in C^0_0$ such that $$d(\phi_t(x),\phi_{h(t)}(y))<\delta_T\|X(\phi_t(x))\|$$ for every $t\in\mathbb{R}$ and some $x,y\in X$,
then $h(T)>0$. Arguing by contradiction, if the claim is not true, then for any $n>0$, 
 there exist two regular points $x_n,y_n\in M$ and a continuous map $h_n$ of $\R$ with $h_n(0)=0$ such that 
 $$d(\phi_t(x_n), \phi_{h_n(t)}(y_n))<\frac{1}{n} \|X(\phi_t(x_n))\|, \forall t\in\mathbb{R},$$
 and $h_n(T)<0$. Take $n$ sufficiently large such that $1/n\leq \gamma_T e^{-LT_0}(2(e^{2LT_0}+1))^{-1},$ where $\gamma_T$ is the constant given Corollary \ref{separation1} and $L$ is the Lipschitz constant of the flow. Here $\gamma_T$ is taken to be smaller than the constant $c$ given in Lemma \ref{2}. 
 Then we have that 
 \begin{eqnarray*}d(\phi_t(y_n), \phi_{h_n(t)}(y_n))&\leq &d(\phi_t(y_n),\phi_t(x_n))+ d(\phi_t(x_n), \phi_{h_n(t)}(y_n))\\
 &\leq & \frac{1}{n}e^{L|t|}\|X(x_n)\|+\frac{1}{n}\|X(\phi_t(x_n))\|\\
 &\leq& \frac{1}{n}e^{2L|t|}\|X(\phi_t(x_n))\|+\frac{1}{n}\|X(\phi_t(x_n))\|\\
 &\leq& \gamma_Te^{-LT_0}\|X(\phi_t(y_n))\|,
 \end{eqnarray*}
 for any $-T_0\leq t\leq T_0$. 
Since $h_n(T)<0,$  we have $T\leq |h_n(T)-T|$. Besides, we know that $h_n(0)=0$. Hence, by continuity, there exists $s\in (0,T]$ such that 
 $|h_n(s)-s|=T$. Then $$d(\phi_s(y_n), \phi_{h_n(s)}(y_n))\leq \gamma_Te^{-LT_0}\|X(\phi_s(y_n))\| \leq \gamma_T\|X(\phi_s(y_n))\|,$$
 which contradicts Corollary  \ref{separation1}. Hence we finish the proof of the claim.

Now let's show moreover that $h((n+1)T)-h(nT)>0$, for any $n$. Define 
$$h'(t)=h(t+nT)-h(nT).$$
This gives $h'(0)=0$ and thus $h'\in C_0^0$. From $d(\phi_t(x),\phi_{h(t)}(y))<\delta_T\|X(\phi_t(x))\|$, it follows that
$$d(\phi_{t-nT}(\phi_{nT}(x)),\phi_{h(t)-h(nT)}(\phi_{h(nT)}(y)))<\delta_T\|X(\phi_{t-nT}(\phi_{nT}(x)))\|,$$
for any $t\in\R$. Replacing $t-nT$ by $s$ gives us
$$d(\phi_{s}(\phi_{nT}(x)),\phi_{h'(s)}(\phi_{h(nT)}(y)))<\delta_T\|X(\phi_{s}(\phi_{nT}(x)))\|,$$
for any $s\in\R$. Applying the above claim to $h'$, we obtain that $h'(T)>0$, which implies that
$$h((n+1)T)-h(nT)>0.$$

Besides, for any $\vep>0$, by the continuity of the vector field, there exists a constant $0<T<\frac{T_0}{3}$ such that
$$
d(x,\phi_s(x))\leq \frac{\vep}{2}\|X(x)\|, \text{ for any } |s|\in [0,T] \text{ and }x\in M. 
$$
For this $T$, we fix $\delta_T$ and $\tau_T$ as above and also
$0<\delta<\min\left(\delta_T,\frac{\vep}{2}e^{-LT}\right)$.

For $y\in \Gamma_\delta(x)$, there is $h\in C^0_0$ such that
$d(\phi_t(x),\phi_{h(t)}(y))<\delta \|X(\phi_t(x))\|$ for every $t\in\mathbb{R}$. As $\delta<\delta_T$, we have
$$d(\phi_t(x),\phi_{h(t)}(y))<\delta_T\|X(\phi_t(x))\|$$ for every $t\in\mathbb{R}$. 
According to the choice of $\delta_T$, we have that
$h((n+1)T)-h(nT)>0$ for every $n\in \Z$.
Now let's define an increasing homeomorphism $\hat{h}: \mathbb{R}\to \mathbb{R}$ by $$\hat{h}(nT)=h(nT)$$ ($\forall n\in\mathbb{Z}$) and extend by linearity to $[nT,(n+1)T]$ for every $n\in\mathbb{Z}$.
From the properties of $h$, we know that $\hat{h}(0)=0$ and that $\hat{h}$ is an increasing homeomorphism. 
Moreover, for every $t\in [nT,(n+1)T]$ (for some integer $n$) there is $t'\in [nT,(n+1)T]$ with $\hat{h}(t)=h(t')$ and so
\begin{eqnarray*}
d(\phi_t(x),\phi_{\hat{h}(t)}(y))&=&d(\phi_t(x),\phi_{h(t')}(y))\\
&\leq& 
d(\phi_t(x),\phi_{t'}(x))+d(\phi_{t'}(x),\phi_{h(t')}(y))\\
&\leq& \frac{\vep}{2}\|X(\phi_{t}(x))\|+\delta\|X(\phi_{t'}(x))\|\\
&\leq& \frac{\vep}{2}\|X(\phi_{t}(x))\|+\delta e^{LT}\|X(\phi_{t}(x))\|\\
&=&\vep \|X(\phi_{t}(x))\|,
\end{eqnarray*}
for any $t\in \R$. Hence the proof is complete. 
\end{proof}

\begin{corollary}
\label{c1}
Let $\phi$ be a flow generated by a $C^1$ vector field $X$ on a compact space $M$. Then there exists a constant $c>0$ such that for every $0<\vep\leq c$, there is $\delta>0$ such that
$\Gamma_\delta(x)\subset \Gamma^{\ast}_\vep(y)$ for every regular points $x,y\in X$ with $y\in \Gamma_\delta(x)$.
\end{corollary}

\begin{proof}
For any $\vep>0$, by Lemma \ref{increasing}, there exists a constant $\delta>0$ such that 
if $y\in \Gamma_\delta(x)$ and $z\in \Gamma_\delta(x)$ there are increasing homeomorphisms
$\hat{h},\tilde{h}\in C^0_0$ such that
$d(\phi_t(x),\phi_{\hat{h}(t)}(y))\leq \frac{\vep}{4}\|X(\phi_t(x))\|$ and $d(\phi_t(x),\phi_{\tilde{h}(t)}(z))\leq\frac{\vep}{4}\|X(\phi_t(x))\|$
for every $t\in \mathbb{R}$.
It follows that
$d(\phi_{\hat{h}(t)}(y),\phi_{\tilde{h}(t)}(z))\leq\frac{\vep}{2} \|X(\phi_t(x))\|$ for all $t\in \mathbb{R}$. 
Let $c$ be the constant given in Lemma \ref{2}. By Lemma \ref{2}, it follows that
$$d(\phi_{\hat{h}(t)}(y),\phi_{\tilde{h}(t)}(z))\leq\vep  \|X(\phi_{\hat{h}(t)}(y))\|,$$ for all $t\in \mathbb{R}$. 
Replacing $t$ by $\hat{h}^{-1}(t)$, we obtain that
$$
d\left(\phi_t(y),\phi_{(\tilde{h}\circ \hat{h}^{-1})(t))}(z)\right)\leq\vep \|X(\phi_t(y))\|,
$$
for all  $t\in \mathbb{R}$. Since both $\tilde{h}$ and $\hat{h}^{-1}$ are increasing homeomorphisms preserving $0$, $\tilde{h}\circ \hat{h}^{-1}$ is also a continuous increasing homeomorphim preserving $0$, we conclude that $z\in \Gamma^{\ast}_\vep(y)$, which finishes the proof. 
\end{proof}

\subsection{``Almost identical" reparametrizations}
The following proposition is about an ``almost identical" type of control on the reparametrizations. Based these control, we will obtain the relation between the dynamical balls $\Gamma_{\vep}(x)$ and $\Gamma_{\vep}^{\al}(x)$ and furthermore the relation between the dynamical balls $\Gamma^{\ast}_{\vep}(x)$ and $\Gamma_{\vep}^{\al,\ast}(x)$.
We note that a similar result for increasing and hence positive reparametrizations on $[0,T]$ was given in \cite{WangWen} (Proposition 1). In contrast, our setting allows for reparametrizations that are merely in $C_0^0$, potentially decreasing and taking negative values. 
\begin{Prop}\label{4}
Let $\phi$ be a flow generated by a $C^1$ vector field $X$ on a compact space $M$.  Then, there exists a constant $T_0$ such that for any $\lambda>0$ and $0<b<T_0$, there is $\vep>0$ such that   for any two regular points $x,y\in M$ and any $T\in[b,T_0]$, if  there exists $h\in C^0_0$ with $h(0)=0$ such that $$d(\phi_t(x),\phi_{h(t)}(y))\leq \vep \|X(\phi_t(x))\|,$$for any $t\in[0,T]$, then one has $|h(t)-t|\leq \lambda t,$ for any $t\in[b,T].$
\end{Prop}

\begin{proof}Let $L>0$ be a local Lipschitz constant of $X$.  By the continuous dependence of the solution on the initial value again,  for any regular points $x\in M\setminus{\rm Sing}(X)$ and $t\in \mathbb{R}$, we have $$e^{-L|t|}\|X(x)\|\leq\|X(\phi_{t}(x))\|\leq  e^{L|t|}\|X(x)\|.$$ Moreover, for $x,y\in M$, we have $e^{-L|t|}d(x,y)\leq d(\phi_{t}(x),\phi_{t}(y))\leq e^{L|t|}d(x,y)$. Let $T_0$ be the constant given in Lemma \ref{separation0}.  Fix $\lambda>0$ and $0<b<T_0$. Let $b$ be the constant $T$ in Lemma \ref{separation0}. Then there exists $\gamma_b$ such that 
$$d(\phi_s(x),x)\geq \gamma_b \|X(x))\|,$$ for any $s\in [b,T_0]$. 
Let $\vep=\min\{\frac{\gamma_b}{2(1+e^{2LT_0})}, \frac{\lambda b}{6(1+e^{2LT_0})}, e^{-2LT_0}c\}$, where $c$ is the constant given in Lemma \ref{2}.  Consider any two regular points $x,y\in M$ and $T\in[b,T_0]$ such that there exists $h\in C^0_0$ satisfying 
   $$d(\phi_t(x),\phi_{h(t)}(y))\leq \vep \|X(\phi_t(x))\|,$$
 for any $t\in[0,T]$.
  Now we show that $|h(t)-t|\leq \lambda t,$ for any $t\in[b,T].$

First of all, by the choice of $\vep$,  we have 
$$d(\phi_{h(t)}(y),\phi_{t}(y))\leq d(\phi_{h(t)}(y),\phi_{t}(x))+d(\phi_{t}(x),\phi_{t}(y)
\leq 2(1+e^{2LT_0})\vep\|X(\phi_{h(t)}(y))\|\leq \gamma_b\|X(\phi_{h(t)}(y))\|,$$
for any $t\in [b,T]$. Base on this, Lemma \ref{separation0} and the continuity of $|h(t)-t|$, we get that 
$$|h(t)-t|<b<T_0,$$ for any $t\in [b,T]$. Moreover, by Lemma \ref{3}, 
$$|h(t)-t|<6(1+e^{2LT_0})\vep\leq \lambda b<\lambda t,$$
for any $t\in [b,T]$. Hence the proof is complete. 
\end{proof}
Using Proposition~\ref{4}, one can extend the control on the reparametrizations to the time infinity. The proof is the same as Lemma 4 in \cite{WangWen} and thus we skip the details. 
\begin{Lem}\label{almostlinear}
Let $\phi$ be a flow generated by a $C^1$ vector field $X$ on a compact space $M$. There exists a constant $T_0$  such that for any $\lambda>0$ and $0<b<T_0$, there is $\vep>0$ such that   for any $x,y\in M\setminus {\rm Sing}(X)$  and any $T\in[b,+\infty)$ if there exists $h\in C^0_0$  such that $$d(\phi_t(x),\phi_{h(t)}(y))\leq \vep \|X(\phi_t(x))\|$$ for any $t\in[0,T]$
holds,
then $|h(t)-t|\leq \lambda t,$ for any $t\in[b,T].$
\end{Lem}

Based on Proposition \ref{4} and Lemma \ref{almostlinear}, we obtain the following proposition, which tells us essentially $\Ga_\vep(x)$ and $\Ga_\vep^\al(x)$ are the same for regular points $x$. 
\begin{proposition}\label{cor_equal_balls}
   Let $\phi$ be a flow generated by a $C^1$ vector field $X$ on a compact space $M$. Then for every $\al \in (0,1)$ there exists $\vep_0>0$ such that for every $\vep \in (0,\vep_0]$ and $x \in M$ we have $\Ga_\vep(x) = \Ga_\vep^\al(x)$.
\end{proposition}
\begin{proof}First of all, by definitions we know that $\Ga_\vep^\al(x)\subset \Ga_\vep(x)$. So, we only need to show the other direction, i.e., $\Ga_\vep(x) \subset \Ga_\vep^\al(x)$. For any $y\in  \Ga_\vep(x)$, let $h$ be in $C^0_0$ such that $$d(\phi_t(x),\phi_{h(t)}(y))\leq \vep \|X(\phi_t(x))\|,$$ for any $t\in\R.$ Now let's find $h'\in Rep(\al)$ such that 
$d(\phi_t(x),\phi_{h'(t)}(y))\leq \vep \|X(\phi_t(x))\|,$ for any $t\in\R.$

First of all, for any $t\in \R$, we define a new map $h_t(s):=h(s+t)-h(t)$. It follows that $h_t(0)=0$ and thus $h_t(s)\in C_0^0$ for any $t$. 
Moreover, we have
 $$d(\phi_s(\phi_{t}(x)),\phi_{h_t(s)}(\phi_{h(t)}(y)))= d(\phi_{s+t}(x),\phi_{h(s+t)}(y))
 \leq \vep \|X(\phi_{s+t}(x))\|=\vep\|X(\phi_s(\phi_t(x)))\|$$ for any $s\in\R.$  Apply Lemma \ref{almostlinear} to $h_t(s)$ and the constant $\al$, for any $0<b<T_0$, there is $\vep>0$ such that $|h_t(s)-s|\leq \al s,$ i.e.,
$|\frac{h_t(s)}{s}-1|\leq \al,$
for any $s\in[b,\infty).$ 
Take $s'=s-t$ with $s'=s-t\geq b$.  We have that 
$|\frac{h_t(s')}{s'}-1|\leq \al,$ which gives us
$$|\frac{h_t(s-t)}{s-t}-1|=|\frac{h(s)-h(t)}{s-t}-1|\leq \al, \text{ for any }s-t\geq b.$$

Secondly, let us modify $h$. Define a new function $h':\R\rightarrow \R$ as following:
\begin{enumerate}
\item $h'(kb)=h(kb)$, for all $k\in\Z$;
\item For any $t\in [kb,(k+1)b]$, $h'(t)$ is the linear map connecting $h'(kb)$ and $h'((k+1)b)$, i.e., 
$$h'(t)=(h((k+1)b)-h(kb))t/b+(1+k)h(kb)-kh((k+1)b).$$ 
\end{enumerate}
It is easy to see that $h'$ is still an increasing homeomorphism and $h'(0)=0$. 
Consider any two points $t=k_1b+s_1$ and $s=k_2b+s_2$, where $s_1,s_2\in [0,b)$. 
When $k_1=k_2$, we have
$$\frac{h'(s)-h'(t)}{s-t}=\frac{h((k_1+1)b)-h(k_1b)}{b}\leq \al.$$
By the choice of $\vep$, we have
$$|\frac{h((k+1)b)-h(kb)}{b}-1|\leq \al,$$ for any $k$. Moreover,  
when $k_1\neq k_2$, 
$$\frac{h'(s)-h'(t)}{s-t}\leq \frac{h'((k_2+1)b)-h'(k_1b)}{(k_2-k_1)b}\leq \frac{1}{k_2-k_1}\sum_{k_1}^{k_2}\frac{h'((i+1)b)-h'(ib)}{b}\leq \al.$$
The proof is complete. 
\end{proof}
\begin{corollary}\label{boost}If $\phi$ is a flow with $C^1$ vector field $X$ on $M$, then for every $\al \in (0,1)$ there exists $\vep_0>0$ such that for every $\vep \in (0,\vep_0]$ and $x \in M$ we have $\Ga^{\ast}_\vep(x) = \Ga_\vep^{\al,\ast}(x)$.
\end{corollary}
\subsection{Equivalent definitions of (positively) rescaled expansive measures}
In this subsection, we state the proof of Theorem \ref{equivalent}, which ensures that using $Rep(\al)\cap Rep$ in the definition of rescaled expansiveness, rather than the space of $C^0_0$
reparametrizations, imposes no additional restriction.
\begin{proof}[The proof of Theorem \ref{equivalent}]
   By the definitions, we know that $\Ga^{\al,\ast}_\vep(x)\subset \Ga_\vep(x)$ and hence \eqref{thmA_item1}$\Rightarrow$\eqref{thmA_item2}. We only need to prove that \eqref{thmA_item2}$\Rightarrow$\eqref{thmA_item1}. Fix any $\al\in(0,1)$. Assume that there exists $\vep>0$ such that $\mu(\Ga^{\al,\ast}_\vep(x))=0$ for $\mu$-almost every $x \in M$.  Denote the collection of these points as $A_{\vep}$. Due to the fact that $\mu(\Ga^{\al,\ast}_{\vep_1}(x))\leq  \mu(\Ga^{\al,\ast}_{\vep_2}(x))$ for any $\vep_1\leq \vep_2$, we can assume that $\vep$ is sufficiently small that it is smaller than the constant $\vep_0$ given by Proposition \ref{cor_equal_balls} and Corollary \ref{boost}. Hence we have $\Ga^{\al,\ast}_\vep(x) = \Ga^{\ast}_\vep(x)$ for every regular points $x \in M$.
 Arguing by contradiction,  assume that $\mu$ is not rescaled expansive, i.e., for any integers $n$, there exists a point $x_n\in M$ such that $\mu(\Ga_{1/n}(x_n))>0$. 
By the choice of $\vep$ and $A_{\vep},$  $\Ga_{1/n}(x_n) \cap A_\vep\neq\emptyset$. 
Hence, there exists $y_n\in \Ga_{1/n}(x_n)$ such that $\mu(\Ga^{\ast}_{\vep}(y_n))=0$. Let $\delta$ be the constant given in Corollary \ref{c1}. Let $n$ sufficiently large so that $1/n\leq \delta$. By Corollary \ref{c1}, 
$$\mu(\Gamma_{1/n}(x_n))\leq \mu(\Gamma_\delta(x_n))\leq \mu( \Gamma^{\ast}_\vep(y_n))=0,$$ which is a contradiction. 
\end{proof}

A stronger version of rescaled expansive measures can be similarly defined by just considering the positive direction. 
\begin{definition}\label{positive}
 For a flow $\phi$  generated by a $C^1$ vector field  $X$ on a compact manifold  $M$,  {\it the positive rescaled dynamical ball} centered in $x \in M$ with radius $\vep>0$ is defined as
 $$\Ga^{+}_\vep(x) = \{y \in M : \exists h \in \SC^0_0 \ \text{s.t.}\ d(\phi_s(x), \phi_{h(s)}(y))\leq \vep \|X(\phi_s(x))\| \enspace \forall s \in \R^{+}\}$$
 where $C^0_0$ is the set of continuous functions $h:\R\rightarrow\R$ with $h(0)=0.$
A Borel probability measure $\mu$ on $M$ is called {\it positive rescaled expansive} for the flow $\phi$ if there exists $\vep > 0$ such that $\mu(\Gamma^+_\vep(x)) = 0$ for every $x \in M$.
\end{definition}
It is easy to see that a positively rescaled expansive measure is always rescaled expansive. The following example tells us the fact that a rescaled expansive measure may not be positively rescaled expansive. 
\begin{example}Consider the two dimensional sphere with the classical map $f_2: z\mapsto z^2$. Let $\mu_{\delta}$ be the Lebesgue measure supported on the circle $S_{\delta}=\{z\vert \|z\|=\delta\}$. It is easy to see that 
$\mu_{\delta}$ is rescaled expansive, but not positively rescaled expansive as long as $\delta<1$.  
\end{example}

Using similar arguments, we obtain the following version of Theorem \ref{equivalent} for positive rescaled dynamical balls. 
\begin{theorem}\label{thmA_forward}
       Let $\phi$ be a flow with $C^1$ vector field $X$ on $M$ and $\mu$ is invariant. Then $\mu$ is positively rescaled expansive if and only if there exists $\vep>0$ such that $\mu(\Ga^+_\vep(x))=0$ for $\mu$-almost every $x \in M$.
\end{theorem}

\section{The lower bound of the decay rate of the rescaled Bowen balls}
This section was devoted to showing  the existence of the invariant rescaled expansive measures when the flow has positive entropy, i.e., the proof of Theorem \ref{entropythm}. The crucial step is to  establish an estimation on the lower bound of the decay rate of  Bowen balls. Different types of  Bowen balls for flows were introduced in several papers, such as, \cite{Sun, WangWen}, among others. Here we recall the rescaled Bowen balls defined in \cite{WangWen}.
\begin{Def}
Let  $\phi$ be a flow generated by a $C^1$ vector field $X$ on a compact manifold  $M$.
For any regular point $x\in M$, $t>0$ and $\vep>0$, the rescaled Bowen balls are defined to be
$$B_1^*(x,t,\vep)=\big\{y\in M|\: d(\phi_s(x),\phi_s(y))<\vep\|X(\phi_s(x))\|,~\textrm{for any $s\in[0,t]$}\big\}.$$
Similarly, one can consider the following rescaled Bowen balls:
$$B^*_2(x,t,\vep)=\big\{y\in M|\exists h\in Rep, s.t. d(\phi_{h(s)}(x),\phi_s(y))<\vep\|X(\phi_{h(s)}(x))\|,~\textrm{for any $s\in[0,t]$}\big\}$$
and
$$B^*_3(x,t,\vep)=\big\{y\in M|\exists h\in Rep, s.t. d(\phi_s(x),\phi_{h(s)}(y))<\vep\|X(\phi_s(x))\|,~\textrm{for any $s\in[0,t]$}\big\}.$$
\end{Def}
It is easy to see that $B_1^*(x,t,\vep)\subset B^*_2(x,t,\vep)$ and that $B_1^*(x,t,\vep)\subset B^*_3(x,t,\vep)$. Besides, the following lemmas tell us the further relation among these three rescaled Bowen balls.  
\begin{lemma}[Lemma 5 in \cite{WangWen}]\label{relation}Let  $\phi$ be a flow generated by a $C^1$ vector field $X$ on a compact manifold  $M$.  There exists a constant $T_0$ such that
given $\lambda>0$ and $b\in(0,T_0)$, there exists $\vep_0>0$ such that for any $x\in M\setminus {\rm Sing}(X)$, any $\vep\in(0, \vep_0]$ and any $t\geq b$, one has
$$B_3^*(x, t, \vep)\subset B_2^*(x, (1-\lambda)t, \vep) \text{ and }B_2^*(x, t, \vep)\subset B_3^*(x, (1-\lambda)t, \vep).$$
\end{lemma}
The following lemma is not stated but essentially covered in  the proof of Lemma 6 in \cite{WangWen}. 
\begin{lemma}[Lemma 6 in \cite{WangWen}]\label{lemma: coverballs}Let  $\phi$ be a flow generated by a $C^1$ vector field on a compact manifold  $M$.  There exist constants $T_0, c, L$ such that for any $0<T<T_0$ and any $0<\delta<c$, there is $\vep>0$ such that for any $t\geq L\geq T$, we can find at most $3^{[\frac{t}{L}]}$ points $\{y_j\}$ such that
$$B_2^*(x,t,\vep)\subset \bigcup_{j}B_1^{\ast}(y_j, t, \delta).$$

\end{lemma}

\begin{lemma}\label{rescaled123} 
Let  $\phi$ be a flow generated by a $C^1$ vector field on a compact manifold  $M$ with a Borel probability measure $\mu$. Then, for any $x\in M$, we have
$$\lim_{\vep \to 0} \limsup_{t \to \infty} \frac{-\log(\mu(B_3^*(x, t, \vep)))}{t}= \lim_{\vep \to 0} \limsup_{t \to \infty} \frac{-\log(\mu(B_2^*(x, t, \vep)))}{t}\leq \lim_{\vep \to 0} \limsup_{t \to \infty} \frac{-\log(\mu(B_1^*(x, t, \vep)))}{t}.$$
Similar results also hold for $\lim_{\vep \to 0} \liminf_{t \to \infty}$.
\end{lemma}
\begin{proof}The first equality is a direct corollary from Lemma \ref{relation}. The second inequality follows directly from the definitions since $B_1^*(x,t,\vep)\subset B^*_2(x,t,\vep)$.
\end{proof}

The following theorem gives us a lower bound of the decay rate of the rescaled Bowen balls. The classical Shannon–McMillan–Breiman theorem and Borel-Cantelli Lemma play a crucial role in establishing the relationship between the decay rate of rescaled Bowen balls and the entropy. 
\begin{theorem}\label{11}
 Let  $\phi$ be a flow generated by a $C^1$ vector field $X$ on a compact manifold  $M$ with a Borel probability measure $\mu$ such that $\mu( {\rm Sing}(X))=0$. If $\mu$ is ergodic, we have 
$$\lim_{\vep \to 0} \liminf_{t \to \infty} \frac{-\log(\mu(B_i^*(x, t, \vep)))}{t} \geq h_{\mu}(\phi),$$ where $i=1,2$ or $3$. 
\end{theorem}
\begin{proof}According to Lemma \ref{relation} and Lemma \ref{rescaled123}, we only need to prove the claim for  $B_2^*(x, t, \vep)$.  We only state the proof for the case when $h_{\mu}(\phi)$ is finite and the case for the entropy to be infinite is similar.  First of all, since $\mu$ is an ergodic measure for the flow, it is also an ergodic measure for the time $\tau$  map $\varphi_\tau$ of the flow except countably many time. Hence, we can choose $\tau\neq 0$ such that $\mu$ is a $\phi_{n\tau}$ ergodic measure and the same for the iterations $\phi_{n\tau}$, for any $n\in\N$. 

Let $L\in \N$ such that $L\geq T_0,$ where $T_0$ is given in Lemma  \ref{lemma: coverballs}. 
Let $\delta>0$ be an arbitrary small number. By the definition of measure-theoretic entropy, there exists an arbitrary finite measurable partition  $\tilde{\xi}=\{B_{1},\ldots,B_{m}\}$ on $M$ such that 
$$H_{\mu}(\phi_{\tau L}, \tilde{\xi})\leq h_{\mu}(\phi_{\tau L})+\delta/2,$$
where $H$ indicates the classical definition of measure-theoretic entropy for a partition.  

Next let us a classical approximation to this partition using compact subsets. For any $0<\vep'<\frac{\delta}{2m\log m}$, there is a compact subset $A_{j}\subset B_{j}$ such that $\mu(B_{j}\setminus A_{j})<\vep'$ for any $1\leq j\leq m$. Define a new finite measurable partition as $\xi=\{A_{1},\ldots,A_{m},A_{m+1}\}$, where $A_{m+1}$ is the complementary of $\bigcup_{i=1}^{m}A_{i}$. Then
$$H_{\mu}(\varphi_{\tau L},\tilde{\xi})\leq H_{\mu}(\varphi_{\tau L},\xi)+\vep ' m\ln m\leq h_{\mu}(\phi_{\tau L})+\delta.$$
Since $A_1, A_2, \cdots, A_m$ are all compact, the minimal distance between each two of them, namely,   $$\delta_0=\min\{d(x,y)| x\in A_i, y\in A_j, 1\leq i\neq j\leq m\}$$
is positive. 

Let $\delta_1,\delta_2, \delta_3$ be four numbers such that $\delta_1\leq C\delta_0/2$, where $C=\sup_{x\in M}\|X(x)\|$, $0<\delta_2<\delta_3.$

By Lemma \ref{lemma: coverballs}, 
for $\delta_1$,  there is $\vep>0$ such that for sufficiently any large $L$ and $t=nL\tau \geq L$, we can find at most $3^{n\tau}$ points $\{y_j\}$ such that
$$B_2^*(x,n\tau L,\vep)\subset \bigcup_{j}B_1^{\ast}(y_j, n\tau L, \delta_1),$$
  for all regular $x$ and any $n\geq \frac{1}{\tau} \in \mathbb{N}$.

Denote by
$$\xi_{n}=\xi\vee \varphi_{L\tau}^{-1}\xi\vee\cdots\vee\varphi_{L\tau}^{-n+1}\xi.$$

We claim that for any rescaled Bowen ball $B_1^*(x, n\tau L, \delta_1)$, there exists at most $2^n$ many different $P\in \xi_{n},$
such that $P\cap B_1^*(x, n\tau L, \delta_1)\neq\emptyset.$ 
Arguing by contradiction, assume that there are more than $2^n$ such $P$. Then, there exists two different subsets  $P_1, P_2\in \xi_{n}$, such that 
$$P_1\cap B_1^*(x, n\tau L, \delta_1)\neq\emptyset, P_2\cap B_1^*(x, n\tau L, \delta_1)\neq\emptyset,$$
and  there exists $k$ such that
$\phi^{k\tau L}(P_1)\subset A_{j_1}$ and $\phi^{k\tau L}(P_2)\subset A_{j_  2}$ while $j_1\neq j_2$ and $j_1,j_2\in[1,m].$
Take any two points $y_1\in P_1\cap B_1^*(x, n\tau L, \delta_1)$ and $y_2\in P_2\cap B_1^*(x, n\tau L, \delta_1)$. It follows that
$\phi_{k\tau L}(y_1)\in A_{j_1}$ and $\phi_{k\tau L}(y_2)\in A_{j_2}$. By the definition of $B_1^*(x, n\tau L, \delta_1)$, 
one has
$$d(\phi_{k\tau L}(x), \phi_{k\tau L}(y_1))<\delta_1\|X(\phi_{k\tau L}(x))\|<\delta_0/2, $$
$$ d(\varphi_{k\tau L}(x), \phi_{k\tau L}(y_2))<\delta_1\|X(\phi_{k\tau L}(x))\|<\delta_0/2.$$
Thus we have $d(\phi_{k\tau L}(y_1), \phi_{k\tau L}(y_2))<\delta_0$, contradicting with the choice of $\delta_0$. Hence we finish the claim. 
 
By the above claim and the choice of $\vep$, there are at most $3^{n\tau}\cdot 2^{n}$ many $P\in \xi_{n},$
such that $$P\cap B_2^*(x, n\tau L, \vep)\neq\emptyset.$$ 

Define $\xi_{n}(x)$ as the element that contains $x$ in the partition $\xi_{n}$. By Shannon-McMillan-Breiman theorem, we have
 $$\lim\limits_{n\to \infty}\frac{1}{n}\log\mu(\xi_{n}(x))=H_{\mu}(\phi_{\tau L},\xi),$$ for  a.e. $x\in M.$ From this we get that for $\delta$ and $\delta_2$, there exists an integer $N$ such that 
 $\mu(E_N)>1-\delta,$ where
$$E_N=\{x\in M\vert \mu(\xi_n(x))\leq e^{-n(H_{\mu}(\phi_{\tau L},\xi)-\delta_2)}, \forall n\geq N\}.$$ 

 Moreover, for $\delta_3$, let's define the following subset
$$G_n=\{x\in E_N\vert \mu(B_2^*(x, n\tau L, \vep))>3^{2n\tau}\cdot 2^{2n} e^{-n(H_{\mu}(\phi_{\tau L},\xi)-\delta_3)}
\},$$
which is the collection of bad points that we want to avoid. We want to control the measure of this subset up to an exponentially small scale. 
For any $x\in G_n$, consider the subset $P\in \xi_{n},$
such that $$P\cap B_2^*(x, n\tau L, \vep)\neq\emptyset.$$ 
By the claim above, we know that there are at most  $3^{n\tau}\cdot 2^{n}$ such $P$. 
Since $x\in G_n$, we have $$ \mu(B_2^*(x, n\tau L, \vep))>3^{2n\tau}\cdot 2^{2n} e^{-n(H_{\mu}(\phi_{\tau L},\xi)-\delta_3)}.$$
This tells us there are at least one of the $P$ with non-empty intersection with $B_2^*(x, n\tau L, \vep)$ having measure greater than $3^{n\tau}\cdot 2^{n}e^{-n(H_{\mu}(\phi_{\tau L},\xi)-\delta_3)}$. Since $\mu$ is a probability measure, the number of such $P$ can not exceed 
  $3^{-n\tau}\cdot 2^{-n}e^{n(H_{\mu}(\phi_{\tau L},\xi)-\delta_3)}$.

Define $K_n$ to be the total number of $P\in \xi_{n},$
such that $P\cap G_n\neq\emptyset.$ Then, we have the following control 
$$K_n\leq 3^{n\tau}\cdot 2^{n}\cdot 3^{-n\tau}\cdot 2^{-n}e^{n(H_{\mu}(\phi_{\tau L},\xi)-\delta_3)}\leq e^{n(H_{\mu}(\phi_{\tau L},\xi)-\delta_3)}.$$
Denote by $Y_n$ to be the union of $P\in \xi_{n}$
such that $P\cap G_n\neq\emptyset.$   By the definition of $E_N$, we have the following control on the measure of $Y_n$:
$$\mu(Y_n)\leq K_n\cdot e^{-n(H_{\mu}(\phi_{\tau L},\xi)-\delta_2)}\leq e^{n(H_{\mu}(\phi_{\tau L},\xi)-\delta_3)}\cdot e^{-n(H_{\mu}(\phi_{\tau L},\xi)-\delta_2)}\leq e^{(\delta_2-\delta_3)n}.$$
Since $G_n\subset Y_n$, it follows that $$\mu(G_n)\leq e^{(\delta_2-\delta_3)n},$$ for $n\geq N$, which decays exponentially fast. Hence  $\sum_{n\geq N}^{\infty}\mu(G_n)\leq \infty.$
Finally, by Borel-Cantelli Lemma, for $\mu$-a.e. $x\in E_N$, 
$$\lim_{\vep \to 0} \liminf_{n \to \infty} \frac{-\log(\mu(B_2^*(x, n\tau L, \vep)))}{n\tau L}+\frac{\delta_3+\log 2+\tau\log 3 }{L\tau} \geq \frac{1}{L\tau}H_{\mu}(\phi_{\tau L},\xi)\geq h_{\mu}(\phi)+\frac{\delta}{L\tau} .$$
By the arbitrary choice of the constants, we obtain that  
 for $\mu$-a.e. $x\in E_N$, 
$$\lim_{\vep \to 0} \liminf_{n \to \infty} \frac{-\log(\mu(B_2^*(x, n\tau L, \vep)))}{n\tau L}\geq h_{\mu}(\phi).$$
Since $\eta$ is arbitrary and $\delta$ is arbitrary, it follows that
for $\mu$-a.e. $x\in M$
$$\lim_{\vep \to 0} \liminf_{n \to \infty} \frac{-\log(\mu(B_2^*(x, n\tau L, \vep)))}{n\tau L}\geq h_{\mu}(\phi).$$
Taking the following relation into account, 
$$B_2^*(x, (n+1)\tau L, \vep)\subset B_2^*(x, t, \vep)\subset B_2^*(x, n\tau L, \vep), \forall t\in [n\tau L,(n+1)\tau L],$$
we get that
$$\lim_{\vep \to 0} \liminf_{t \to \infty} \frac{-\log(\mu(B_2^*(x, t, \vep)))}{t} \geq h_{\mu}(\phi),$$
for $\mu$-a.e. $x\in M$, 
which completes the proof. 
\end{proof}

Now we are ready to prove Theorem \ref{entropythm}.     
\begin{proof}[The proof of Theorem \ref{entropythm}.]  Let $\mu \in \SM_\phi(X)$ be an ergodic measure. Arguing by contradiction, suppose that $\mu$ is not rescaled expansive, which implies that for any $\vep>0$, there exists a subset $K_{\vep}$ with $\mu(K_{\vep})>0$ such that for any $x\in K_{\vep}$, 
    $$\mu(\Gamma_{\vep}(x))>0.$$
   By Theorem \ref{11}, there exists a Borel set $H\subset X$ with $\mu(H)=1$ such that for every $x \in H$ it holds that $$\lim\limits_{\vep\to 0}\liminf\limits_{t\to \infty}\frac{-\log(\mu(B_i^*(x, t, \vep))}{t}\geq h_{\mu}(\phi)>0.$$
     Since $\mu(K_{\vep})>0$,  it follows that $H\cap K_\vep \neq \emptyset$. We note here that 
    $\Gamma_{\vep}(x)\subset \Gamma^{\ast}_{\vep_1}(x) \subset B_i^*(x, t, \vep_1),$ for any $t.$
Then $-\log\mu(\Gamma_{\vep}(x))\geq -\log\mu(B_i^*(x, t, \vep_1))$ for any $t>0$. But this implies that
 $$\liminf\limits_{t\to \infty}\frac{-\log(\mu(B_i^*(x, t, \vep_1))}{t}\leq \liminf\limits_{t\to\infty }\frac{-\log(\mu(\Gamma_{\vep}(x))}{t}=0,$$
     contradicting the assumption $h_{\mu}(\phi)>0$, and therefore, $\mu$ must be rescaled expansive.
\end{proof}

\appendix

\section{The rescaled Brin-Katok local entropy formula for flows}

In addition to establishing a lower bound on the decay rate of the rescaled Bowen balls, we also obtain an upper bound under the integrability condition of $\ln\|X\|$. Furthermore, we derive a new formulation of the classical Brin–Katok local entropy formula \cite{BK} for flows generated by $C^1$ vector fields, using rescaled Bowen balls. This refined version offers a natural framework for analyzing local entropy in the continuous-time setting, which extends \cite{JCWZ} from the case without singularities  to the case with singularities.

%


As in the classical proof of the Brin–Katok local entropy formula, the Shannon–McMillan–Breiman theorem plays a crucial role in establishing the relationship between the decay rate of rescaled Bowen balls and the entropy. To apply this theorem, we construct a countable measurable partition with finite entropy, whose diameters under iteration can be controlled by the flow speed. The rescaled Bowen balls are then used to cover the partition elements, allowing a comparison between  the decay rate of rescaled Bowen balls and the entropy via the Shannon–McMillan–Breiman theorem. Following an approach inspired by Mañé~\cite{Mane}, we construct such a partition with finite entropy, where the size of its elements can be uniformly controlled by an integrable function. Similar ideas were used in \cite{WangWen}. 
\begin{theorem}[The rescaled Brin-Katok local entropy formula]\label{thm: BK-formula}
   Let  $\phi$ be a flow generated by a $C^1$ vector field on a compact manifold  $M$ with a Borel probability measure $\mu$.  Assume that $\ln\|X\|$ is $\mu$-integrable.    If $\mu$ is an ergodic $\phi$-invariant measure, then for $\mu$-almost every $x \in M$ we have 
    $$\lim_{\vep \to 0} \liminf_{t \to \infty} \frac{-\log(\mu(B_i^*(x, t, \vep)))}{t}= \lim_{\vep \to 0} \limsup_{t \to \infty} \frac{-\log(\mu(B_i^*(x, t, \vep)))}{t} = h_\mu(\phi),$$
    where $i=1,2$ or $3$. 
\end{theorem}

\begin{proof}By Lemma \ref{relation}, Lemma \ref{rescaled123} and Theorem \ref{11},  we only need to get the upper bound control of the decay rate of $B_1^*(x, t, \vep)$, i.e., 
$$ \lim_{\vep \to 0} \limsup_{t \to \infty} \frac{-\log(\mu(B_1^*(x, t, \vep)))}{t}\leq h_\mu(\phi).$$
First of all, when $\mu$ is ergodic for the flow, there exists $\tau$ such that $\mu$ is ergodic to the time $\tau$ map $\phi_{\tau}$. Without loss of generality, we assume that $\tau=1$. 
For any $x\in M$, $n\in \mathbb{N}$ and
$\varepsilon > 0$, we set
$$B(x,n,\vep,\phi_1)=\{y\in M: d(\phi_{i}(x) , \phi_{i}(y)) < \varepsilon \|X(\phi_i(x))\|, \,\,i\in [0,n)\cap \mathbb{N}\},$$
which is the rescaled $(n,\varepsilon, \phi_{1})$-Bowen ball at $x$ for the time $1$ map $\phi_{1}$.  Let $L$ be the Lipschitz number of the vector fields. For the given $\vep>0$, we can choose $0<\eta<e^{-2L}\vep$. For any regular point $x\in M$, if $d(x,y)<\eta\|X(x)\|$, we have
$$d(\phi_s(x),\phi_s(y))\leq e^{Ls}d(x,y)<e^{Ls}\eta\|X(x)\|\leq e^{2Ls}\eta\|X(\phi_s(x))\|<\vep\|X(\phi_s(x))\|$$
for any $s\in[0,1]$. For any $s\in[0,n]$  take $k\in\{0, 1, \cdots, n\}$ such that $k\leq s<(k+1)$,
then we have
$$d(\phi_s(x), \phi_s(y))=d(\phi_{s-k}(\phi_{k}(x)), \phi_{s-k}(\phi_{k}(y))<\vep\|X(\phi_{s-k}(\phi_{k}(x)))\|.$$
Thus we have $B(x,n,\eta,\phi_{1})\subset  B_1^*(x,n,\vep)$. 
By Lemma 2  in ~\cite{Mane}, for the integrable function $\ln \|X\|$, there exists a countably measurable partition $\xi$ of $M$ such that $H_{\mu}(\xi)<+\infty$ and ${\rm diam}(\xi(x))\leq \eta\ln\|X(x)\|$    for  $\mu \ a.e.x\in M,$ where $\xi(x)$ is the element of $\xi$ that contains $x$.
Let
$$\xi_{-n}=\xi\vee \phi_{1}^{-1}\xi\vee\cdots\vee\phi_{1}^{-n+1}\xi.$$ For $\mu$ $a.e. x\in M$, let $C_n(x)$ be the element of $\xi_{-n}$ that contains $x$. 

For any $y\in C_n(x)$, from definitions, we know
$\phi_{n}(y)\in\xi(\phi_{n}(x))$ and then 
$$d(\phi_{n}(x), \phi_{n}(y))\leq \eta\|X(\phi_{n}(x))\|.$$
It follows that
$C_{n}(x)\subset  B(x,n,\eta,\phi_{1})\subset  B_1^*(x,n,\vep)$.
Hence we have that 
$$\lim_{\vep \to 0} \limsup_{t \to \infty} \frac{-\log(\mu(B_1^*(x, t, \vep)))}{t}\leq  \lim_{n\to\infty}-\frac{1}{n}\ln \mu(C_n(x)).$$
Besides, by Shannon-McMillan-Brieman theorem, we know
$$\lim_{n\to\infty}-\frac{1}{n}\ln \mu(C_n(x))=h_{\mu}(\phi_{1},\xi)\leq h_{\mu}(\phi)\, \textrm{ $\mu$ $a.e.x\in M$}.$$
Hence we have 
 $$ \lim_{\vep \to 0} \limsup_{t \to \infty} \frac{-\log(\mu(B_1^*(x, t, \vep)))}{t}\leq h_\mu(\phi),\textrm{ $\mu$ $a.e.x\in M$},$$
 which completes the proof. 
\end{proof}

\end{document}